\newif\ifELS
\ELStrue
\ifELS
\documentclass[3p,times]{elsarticle}
\usepackage{amsthm} %
\usepackage{amsmath}
\usepackage{amsfonts}
\else
\documentclass{amsart}
\fi

\ifELS
\else
\fi

\makeatletter
\newcommand\mynobreakpar{\par\nobreak\@afterheading}

\makeatother

\usepackage{url}
\usepackage{hyperref}   %

\providecommand\bledenredcontinuous{}
\providecommand\eledenredcontinuous{}
\providecommand\bledenredbounded{}
\providecommand\eledenredbounded{}
\providecommand\bcervenec{}
\providecommand\ecervenec{}
\providecommand\bsrpen{}
\providecommand\esrpen{}

\providecommand\bsrpenxv{}
\providecommand\esrpenxv{}
\providecommand\esrpenxvunskip{}
\providecommand\brijenxv{}
\providecommand\erijenxv{}
\providecommand\brijenxvbounded{}
\providecommand\erijenxvbounded{}
\providecommand\bkoneclistopaduxvH{}
\providecommand\ekoneclistopaduxvH{}

\usepackage{cases}

\newcommand\itemref[1]{(\ref{#1})}
\newtheorem{theorem}{Theorem}[section]

\newtheorem{lemma}[theorem]{Lemma}
\newtheorem{proposition}[theorem]{Proposition}
\theoremstyle{definition}
\newtheorem{definition}[theorem]{Definition}
\newtheorem*{ack*}{Acknowledgement}
\theoremstyle{remark}
\newtheorem{remark}[theorem]{Remark}
\newtheorem*{remark*}{Remark}

\newcommand\R{\mathbb{R}}

\newcommand\N{\mathbb N}

\newcommand\C{\mathcal{C}}

\newcommand\ddd{\varrho}
\DeclareMathOperator{\interior}{int}
\DeclareMathOperator{\cvx}{conv}

\DeclareMathOperator{\dist}{dist}

\DeclareMathOperator{\spt}{supp}

\DeclareMathOperator{\Lip}{Lip}

\newcommand\fcolon{\colon}
\newcommand\setcolon{:}
\providecommand\boundary{\partial}
\newcommand{\closure}{\overline}

\newcommand{\normZ}[1]{\left\|#1\right\|_{Z}}

\newcommand \GG{\mathcal G}
\newcommand \GGGG{\widetilde {\mathcal G}}
\newcommand \HH{\mathcal H}
\newcommand \AAA{{\mathcal A}}

\def\xxxa{November 29, 2015}\ifELS\else\edef\xxxa{\noexpand\relax \xxxa}\fi
\edef\xxxb{\today}\ifx\xxxa\xxxb \AtEndDocument{\label{l:XbezF-in-other-paper}}\fi

\def\JKblue{}    %
\def\eJKblue{}

\begin{document}

\title{Extensions of vector-valued Baire one functions with preservation of~points of~continuity\tnoteref{t1}\tnoteref{tack}}

\tnotetext[t1]{\relax
\ %
\\
{\bf
\copyright 2015. This manuscript version is made available under the CC-BY-NC-ND 4.0 license http://creativecommons.org/licenses/by-nc-nd/4.0/
}
}

\def\AckBody{
 The research leading to these results has received funding
 from the European Research Council under the European Union's Seventh
 Framework Programme (FP7/2007-2013) / ERC Grant Agreement n. 291497.
 \\
 The first author would like to thank the private company RSJ a.s.\ for the support of his research activities.
 The second author was also supported by grants No. 14-07880S of GA\,\v{C}R and RVO: 67985840.
}

\tnotetext[tack]{\AckBody}

\ifELS

\author[am]{M.~Koc\corref{cor1}}
\ead{martin.koc@rsj.com}
\author[a1,a2]{J.~Kol\'a\v{r}}
\ead{kolar@math.cas.cz}

\address[am]{RSJ a.s., Na Florenci 2116/15, 110 00 Praha 1, Czech Republic}
\address[a1]{Institute of Mathematics, Czech Academy of Sciences, \v Zitn\'a 25, 115 67 Praha 1, Czech Republic}
\address[a2]{Mathematics Institute, University of Warwick, Coventry, UK \ (September 2014 -- October 2015)}
\cortext[cor1]{Corresponding author}

\else
\author{M. Koc, J. Kol\'a\v{r}}
\fi

\ifELS\else  \begin{document}  \fi

\begin{abstract}
  We prove an extension theorem
  (with non-tangential limits)
  for vector-valued Baire one functions.
  Moreover, at every point where the function is continuous (or bounded),
  the continuity (or boundedness) is preserved.
  More precisely:
  Let $H$ be a~closed subset of a~metric space $X$ and let $Z$ be a~normed vector space.
  Let $f\colon H\to Z$ be a~Baire one function.
  We show
  that
  there is a~continuous function $g\colon (X\setminus H) \to Z$
  such that,
  for every $a\in \partial H$,
  the non-tangential limit of $g$ at
  $a$
  equals $f(a)$
  and, moreover, if $f$
  is continuous at
  $a\in H$
  (respectively
  bounded in a~neighborhood of
  $a\in H$)
  then the extension $F=f\cup g$
  is continuous at
  $a$
  (respectively
  bounded in a~neighborhood of~$a$).

  We also prove a~result on pointwise approximation of vector-valued Baire one functions
  by a~sequence of locally Lipschitz functions that converges ``uniformly''
  (or, ``continuously'')
  at points where the approximated function is continuous.

  In an accompanying paper (Extensions of vector-valued functions with preservation of derivatives),
  the main result is applied to
  extensions of vector-valued functions
  defined on a~closed subset of Euclidean or Banach space
  with preservation of differentiability,
  continuity and (pointwise) Lipschitz property.
\end{abstract}
\begin{keyword}
 vector-valued Baire one functions\sep extensions\sep non-tangential limit\sep continuity points\sep pointwise approximation\sep
 uniform convergence\sep
 continuous convergence

 \MSC[2010]
 54C20\sep 54C05\sep 26A21\sep 46E40
\end{keyword}

\maketitle

\section{Introduction}

 The purpose of this paper is
 to prove an extension theorem for vector-valued
 Baire one functions.
 The result is directly used in the accompanying paper \cite{KocKolarDer}
 where we
 obtain
 new results on extending vector-valued functions that
 are differentiable (or continuous, Lipschitz, \dots) at some points,
 in a~way that preserves the differentiability (continuity, Lipschitz property, \dots).

 Recall that a~function $f$ is {\em Baire one}
 if it is the pointwise limit of a~sequence of continuous functions.

 \brijenxv
 If $(X,\ddd)$ is a metric space, $a\in X$ and $r>0$, $B(a,r)=B_X(a,r)$
 denotes the open ball $\{x\in X\setcolon \ddd(a,x) < r\}$.
 If $X=Z$ is a normed linear space, we sometimes use also the closed ball
 denoted by
 $\closure B_Z(a, r)$.
 \erijenxv

 Our main result is the following theorem:

\begin{theorem}\label{thm:B1ext}
  Let $(X,\ddd)$ be a~metric space, $H\subset X$ a~closed set, $Z$ a~normed linear space
  and $f\fcolon H\to Z$ a~Baire one function.
  Then there
  exists a~continuous function $g\fcolon (X\setminus H) \to Z$ such that
  \begin{equation}\label{eq:ALP3}
  \lim_ {\substack{x\to a\\ x\in X\setminus H}} \normZ{g(x) - f(a)} \frac{\dist (x,H)}{\ddd(x,a)} = 0
  \tag{NT}
  \end{equation}
  for every $a\in \boundary H$,
  \begin{equation}\label{eq:continuity}
  \lim_ {\substack{x\to a\\ x\in X\setminus H}} g(x) = f(a)
  \tag{C}
  \end{equation}
  whenever $a\in\boundary H$ and $f$ is continuous at $a$,
  and
  \brijenxvbounded
  \begin{equation}\label{eq:boundedness}
      \text{$g$ is bounded on $B(a,r) \setminus H$}
  \tag{B}
  \end{equation}
  whenever
  \bkoneclistopaduxvH
  $a\in \boundary H$ (or even $a\in H$),
  \ekoneclistopaduxvH
  $r\in (0,\infty)\cup\{\infty\}$
  and $f$ is bounded on $B(a, 12 r) \cap H$.
  \erijenxvbounded
\end{theorem}

 In \cite[Theorem~6]{ALP}, the first part of the previous theorem
 (i.e., property~\eqref{eq:ALP3} without properties \eqref{eq:continuity} and \eqref{eq:boundedness})
 was proved in the special cases $Z=\R$
 and $\dim Z < \infty$ (coordinate-wise; see \cite[p.~607]{ALP}).
 In \cite{KZ}, this was extended to include property~\eqref{eq:continuity}.
 Our main contribution is that $Z$ can be an arbitrary normed linear space.

 Properties \eqref{eq:ALP3} and \eqref{eq:continuity} constitute the most important part of Theorem~\ref{thm:B1ext}.
 Other statements (continuity of $g$ and property \eqref{eq:boundedness}) are added since they might
 be useful and do not require much labour.
 The continuity of $g$ is achieved just in the last paragraph of the proof, and from there it is
 obvious how a~higher degree of smoothness can also be achieved when $X$ admits a~linear structure
 and a~smooth partition of unity (cf.\ also
 \cite[Lemma~2.5 (PROVISIONAL REFERENCE)]{KocKolarDer}).
  Property~\eqref{eq:boundedness}
  requires only a~bit more of attention during the proof
  and it is
  actually
  used (together with properties \eqref{eq:ALP3} and \eqref{eq:continuity})
  in the accompanying paper \cite{KocKolarDer}.

 Theorem~\ref{thm:B1ext} is proved in Section~\ref{sec:B1ext-}.
 Its proof depends on Proposition~\ref{prop:Lip} which provides
 an approximation of a~given Baire one
 function $f$ by a~sequence of locally Lipschitz functions that
 converges ``uniformly'' at points of continuity of $f$
 (see property UCPC in Definition~\ref{def:UCPC}).

\smallbreak

\section{Approximation of a~Baire one function by a~sequence of continuous functions}\label{sec:approximation}
  In order to prove property~\eqref{eq:continuity}
  from Theorem~\ref{thm:B1ext}
  we need the following auxiliary notion.
\begin{definition}
    \label{def:UCPC}
   Let $Y$, $Z$ be metric spaces and
   let $h_ n\fcolon Y \to Z$ ($n\in \N$) and $f\fcolon Y \to Z$ be arbitrary functions.
   We say that the pair $ ( \{ h_ n \} , f ) $ has 
   the property of
   {\em uniform convergence at points of continuity} (shortly UCPC)
   if the following holds:
   For every $y_ 0\in Y$ such that $f$ is continuous at $y_ 0$, 
   and for every $\varepsilon >0$,
   there is $k_ 0 = k_ 0 (\varepsilon) \in \N$ and a~neighborhood $U = U (\varepsilon)$ of $y_ 0$
   such that
\begin{equation}\label{eq:hkyfy0}
    \normZ{ h_ k(y) - f(y_ 0) } < \varepsilon
\end{equation}
   for every $k\ge k_ 0$ and $y \in U$.

   If $h_ n\fcolon Y \to Z$ ($n\in \N$) are functions with pointwise limit $f\fcolon Y \to Z$,
   we say that the sequence $\{ h_ n \}$ has 
   the property of
   {\em uniform convergence at points of continuity} (shortly UCPC)
   if the pair $ ( \{ h_ n \} , f ) $ has UCPC.
\end{definition}

\begin{remark}\label{rem:UCPC}
    Property UCPC is probably known and studied in the literature.
  In the terminology of
  \cite[\S~15]{Frink},
  $(\{h_n\}, f)$ has UCPC
  if and only if
  $h_n$ {\em converges continuously} to $f$
  at every $y _ 0 \in Y$ where $f$ is continuous.
 
    Directly related and more general notions are studied by
    Kechris and Louveau \cite{KL}.
    It is easy to show that
    a~sequence $\{h_ k\}$ with pointwise
    limit $f$
    has UCPC if and only if,
    expressed in their notation
    (\cite[p.~211, 212]{KL}),
  $\bigcup_ {\varepsilon >0} P_ {\varepsilon, f}^*
   \supset 
   \bigcup_ {\varepsilon >0} P_ {\varepsilon,     %
				   \{ h_ k\} } '
   ,
  $
  where $P=Y$.
\end{remark}

\begin{remark}
  The Hausdorff's definition of uniform convergence at a~point (\cite[p.~285]{Hausdorff})
  and an easy classical result
  \cite[Theorem IV, p.~285]{Hausdorff}
  are not directly related to the property UCPC.
  The difference is that Hausdorff's definition requires an inequality
  similar
  to \eqref{eq:hkyfy0}
  only for $k=k_0$ instead of $k\ge k_0$.
  There is a~loose connection as
  Proposition~\ref{prop:Lip} below is a~generalization of
  the following corollary of \cite[Theorem IV, p.~285]{Hausdorff}:
  If $f$ is a~Baire one function and $A$ is the set of points of continuity of $f$,
  then there is a~sequence $\{f_n\}$
  \brijenxv
  of continuous functions
  \erijenxv
  pointwise converging to $f$ such that
  $\{f_n\}$ converges (Hausdorff) uniformly
  exactly
  at every point of $A$.
  Proposition~\ref{prop:Lip} includes our version of ``uniform convergence at point'', stronger than the Hausdorff's.
  Furthermore, it is generalized to vector-valued functions
  and strengthened to sequences of locally Lipschitz functions.
\end{remark}

   The following proposition constitutes the core part of the proof of our main result.
   It might be of independent interest.

  \begin{proposition}\label{prop:Lip}
   Let $Y$ be a~metric space and $Z$ a~normed linear space.
   If $f\fcolon Y\to Z$ is a~Baire one function
   then $f$ is the pointwise limit
   of a~sequence
\brijenxvbounded
   $\{f_n\}$
\erijenxvbounded
   of bounded locally Lipschitz functions
   with UCPC
\brijenxvbounded
   such that $\{ f_n \}$
   is uniformly bounded
   on $B_Y(a, r)$
   whenever
   $a\in Y$,
   $r\in (0,\infty)\cup\{\infty\}$
   and $f$
   is bounded on $B_Y(a, 2 r)$.
\erijenxvbounded
  \end{proposition}
   In the case $Z=\R$,
   \cite[p.~605]{ALP} notes that
   the pointwise approximation of $f$
   by bounded Lipschitz functions was established in
   \cite[\S~41, pp. 264--276]{Hausdorff}
\bsrpenxv
   and \cite[Proposition 3.9 and before]{CzaszarLaczkovich}.
\esrpenxv

   \begin{proof}[Proof of Proposition~\ref{prop:Lip}]
   By the definition of Baire one functions,
   we can choose
   a~sequence of continuous functions
   $h_ n\fcolon Y\to Z$ such that $h_ n \to f$ pointwise.

   The proposition now follows from
   Lemma~\ref{l:UCPC} and Lemma~\ref{l:Lipschitz} below.
   \end{proof}

   The following two lemmata allow to replace $\{h_ n\}$
   in the proof of Proposition~\ref{prop:Lip}
   by a~sequence
   with the required properties.
   They are independent
   and Lemma~\ref{l:UCPC} can be skipped by readers not interested
   \bsrpenxv
   in properties \eqref{eq:continuity}
   \brijenxvbounded
   and \eqref{eq:boundedness}
   \erijenxvbounded
   of
   \esrpenxv
   Theorem~\ref{thm:B1ext},
   property UCPC
   \brijenxvbounded
   and the uniform boundedness of $\{f_n\}$.
   \erijenxvbounded

   Both lemmata use the fact that every metric space $Y$
   (as well as its arbitrary subspace)
   is paracompact
   (Theorem of A. H. Stone, see, e.g., \cite[p.~603]{Rudin-paracompact}).
   Note that a~topological space is called {\em paracompact} if every open cover
   of this space has a~locally finite open refinement.

\begin{lemma}\label{l:UCPC}
   Let $Y$ be a~metric space and $Z$ a~normed linear space.
   Given a~function $f\fcolon Y \to Z$ and
   a~sequence of continuous functions $h_ n\fcolon Y\to Z$,
   there exist continuous functions  $\tilde h_ n\fcolon Y \to Z$ such that

   \begin{enumerate}[\textup(a\textup)]
   \item
   $\tilde h_ n(y) \to f(y)$ whenever $y\in Y$ and   $h_ n(y) \to f(y)$;

   \item
   $\tilde h_ n(y) \to f(y)$ whenever $y\in Y$ and $f$ is continuous at $y$;

   \item
   $ ( \{ \tilde h_ n \} , f ) $ has UCPC,
   in particular
   $\{ \tilde h_ n \}$ has UCPC provided $f$ is a~pointwise limit of
   $\{ h_ n \}$.
\bledenredbounded
\eledenredbounded
\end{enumerate}
\end{lemma}
   In a~very special case ($Y$ a~compact
   metric space, $Z=\R$
   and $\{h_ n\}$ a~bounded sequence of continuous functions converging
   pointwise to a~function $f$),
   the lemma follows from the proof of \cite[Theorem~2.3, p.~214--215]{KL}.
   Moreover, in this case,
   functions
   $\tilde h_ n$ are obtained as convex combinations of $\{h_ n\}$.

\begin{proof}
   If $M\subset Y$ is a~set and $\varepsilon >0$, denote
$$
   \interior_ \varepsilon M = \{ y \in Y \setcolon 
              \dist (y, Y\setminus M ) \ge \varepsilon \}
  \subset M
 .
$$
  The set $\interior_ \varepsilon M$ is closed. Recall that $\interior M$ is an open set and denotes the topological interior of $M$.

   For $k\in \N$, 
   set $\HH_ k =\{ B_ Z(z, 2^{-k}) \setcolon z\in Z \}$ and
   $\GGGG_ k =\{ \interior f^{-1} (H) \setcolon
   H \in \HH_ k \}$.
   Then
   $\GGGG_ k$
   is an open cover
   of the (open) set $Y_ k := \bigcup \GGGG_ k \subset Y$.

   Let $\GG_ k$ be an open locally finite refinement
   of $\GGGG_ k$ that covers $Y_ k$.
   It is convenient to observe that $Y_ {k+1}\subset Y_ k$.
\label{NNN:odstraneno H_ k(G)}
   For every $G\in \GG_ k$,
   there is $z_ {k,G}\in Z$
   such that 
\begin{equation}\label{eq:fG}
      f(G) \subset  
        B_ Z( z_ {k,G}, 2^{-k} )
     . 
\end{equation}
   For $y\in Y$, the set 
$$
  \GG_ k(y) 
             := \{  G\in \GG_ k \setcolon y \in G \}
$$
   is finite; if $y\in Y\setminus Y_ k$ then $\GG_ k(y)$ is empty.
   Let 
$$
   \GG_ k^{(j)} (y) 
       = \{  G\in \GG_ k \setcolon y \in \interior_ {1/j} G \}
       \subset \GG_ k(y)
   .
$$
   Note that $Y_ k = \bigcup_ j Y_ k^{(j)}$, where
   $Y_ k^{(j)} = \bigcup \{ \interior_ {1/j} G \setcolon
       G\in \GG_ k 
       \}.$
   Let
\begin{align*}
   \widetilde\Phi_ k (y) & = \bigcap \{ 
                             \closure B_ Z( z_ {i,G}, 2^{-i}  )
                            \setcolon
                             i \in \N \cap [1,k],\ 
                             G\in \GG_ i (y) \}
   & k\in\N,\ y\in Y
   ,
\\
   \widetilde\Phi_ k^{(j)} (y) & = \bigcap \{
                             \closure B_ Z( z_ {i,G}, 2^{-i}  )
                            \setcolon
                             i \in \N \cap [1,k],\ 
                             G\in \GG_ i^{(j)} (y) \}
   & k,j\in\N,\ y\in Y
   ,
\end{align*}
  where we understand $\bigcap \emptyset = Z$.
  Note that $\widetilde\Phi_ k (y) \subset \widetilde\Phi_ k^{(j)} (y)$.
  Also note that
\begin{equation}\label{eq:m}
  \widetilde \Phi_ {k+1}^{(j)}(y) \subset \widetilde \Phi_ k^{(j)}(y)
   \text{\qquad and \qquad }
  \widetilde \Phi_ {k}^{(j+1)}(y) \subset \widetilde \Phi_ k^{(j)}(y)
.
\end{equation}
  Later, we prove that $\widetilde\Phi_ k^{(j)} (y)$ is a~lower semi-continuous
  multivalued map, but we do not make such a~claim for $\widetilde\Phi_ k (y)$.
  Let
$$
    C_ k
    :=\{ y \in Y \setcolon 
                      h_ k(y) \in \widetilde\Phi_ k (y)
   \}
$$
Then $C_ k$ is a~closed set.
   Indeed
   if $y_ 0 \notin C_ k$, then the closed set $F:=\widetilde \Phi_ k (y_ 0)$ 
   does not contain $h_ k(y_ 0)$, and the same is true
   for $h_ k(y)$ with $y$ in a
   neighborhood $U_ 1$ of $y_ 0$.
   Since the elements of the finite set $\GG_ i(y_ 0)$ ($i=1,\dots,k$) are open,
   we have $\widetilde \Phi_ k (y) \subset F$ for $y$ in a~
   neighborhood $U_ 2$
   of $y_ 0$.
   Therefore, $y_0\in U_ 1\cap U_ 2 \subset Y\setminus C_ k$.
   Since $y_0 \notin C_ k$ was arbitrary, 
   $C_ k$ is a~closed set.

   Let
\begin{numcases}{ \Phi_ k^{(j)} (y) = }
    \label{eq:def1}
                     \{ h_ k(y) \}
                     &  if 
                        $ y \in C_ k $
                \\
    \label{eq:def2}
                     \widetilde\Phi_ k^{(j)} (y) 
                     &  otherwise.
\end{numcases}
   Then $\Phi_ k^{(j)} (y)$ is a~closed convex set in $Z$,
   $f(y) \in \widetilde \Phi_ k (y) \subset
             \widetilde \Phi_ k^{(j)} (y)$
   and $\Phi_ k^{(j)} (y)$ contains $f(y)$ or $h_ k(y)$. 
   We have
\begin{equation}\label{eq:PhiSubset}
   \Phi_ k^{(j)}(y) \subset \widetilde \Phi_ k^{(j)} (y)
\end{equation}
   for every $y\in Y$
   independently of which of \eqref{eq:def1}, \eqref{eq:def2}
   happens to be true.

   From the definitions and from \eqref{eq:fG},
\begin{equation}\label{eq:inball-F2}
            \widetilde \Phi_ k^{(j)} (y)
            \subset
                    \closure B_ Z( z_ {k,G},    2^{-k} )
            \subset \closure B_ Z( f(y_ 0), 2^{1-k} )
\end{equation}
   whenever $G\in \GG_k$ and
   $y, y_ 0 \in \interior_ {1/j} G$.

   By the definition of $\widetilde \Phi_ k^{(j)}$ and
   since
   $Y \setminus \interior_ {1/j} G$
   is open
   for every $ G\in \GG_ i $
   and $\GG_ i$ is locally finite ($i=1,\dots, k$),
   we can show that
   the multivalued map
   $\widetilde\Phi_ k^{(j)} \fcolon Y \to Z$ is lower semi-continuous.
   Even more, we show that
   for any set $M\subset Z$,
\begin{equation}\label{eq:hitspoint}
   \left( \widetilde\Phi_ k^{(j)} \right) ^{-1} (M) 
   = \{ y \in Y \setcolon M \cap \widetilde\Phi_ k^{(j)} (y) \neq 
   \emptyset \}
   \qquad
   \text{is open.}
\end{equation}
   This is obviously true if the set in \eqref{eq:hitspoint} is open
   for every singleton
   $M=\{z\} \subset Z$.
   Let $M=\{z\} \subset Z$.
   If $y_0\in Y$ is fixed
      so that $ z \in \widetilde\Phi_ k^{(j)} (y_0)$,
   let
     $U$ be a~neighborhood of $y_0$ such that
     $\GG_ {i,U} := \{ G \in \GG_i \setcolon G \cap U \neq \emptyset \}$
     is finite
     for all $i=1,\dots,k$,
   and let
$$
    H
        \
      :=
        \
        \bigcap _{i=1}^k
        \quad
        \bigcap
          _ {
             G
             \in \GG_ {i,U}
             \, \setcolon \,
             z \notin \closure B_ Z( z_ {i,G}, 2^{-i}  )
            }
        \quad
          Y \setminus \interior_ {1/j} G
    .
$$
   Then $ z \in \widetilde\Phi_ k^{(j)} (y) $
   for every $y \in U \cap H$,
   which is an open neighborhood of $y_0$.
   Thus the set in \eqref{eq:hitspoint} is indeed open.

   Also $\Phi_ k^{(j)}$ is lower semi-continuous.
   Indeed,
   if $O \subset Z$ is open, $y_ 0\in Y$ and $\Phi_ k^{(j)} (y_ 0) \cap O$
   contains a~point $z$, we have two cases to consider:

   1. If $y_ 0 \in C_ k$, i.e.\ $\Phi_ k^{(j)} (y_ 0) = \{ h_ k(y_ 0) \}$,
   we note that 
   $\widetilde \Phi_ k^{(j)} (y_ 0) 
   \supset \widetilde \Phi_ k  (y_ 0)$
   contains $h_ k(y_ 0) = z \in O$.
   Since the set in \eqref{eq:hitspoint} is open with $M=\{z\}$ we have, 
   for $y$ in a~neighborhood $V_ 1$ of $y_ 0$,
   $z \in  \widetilde \Phi_ k^{(j)} (y) \cap O $.
   Since $h_ k$ is continuous, we have  
   $\{ h_ k(y) \} \subset O $
   for $y$ in a~neighborhood $V_ 2$ of $y_ 0$.
   Thus, for $y\in V_ 1\cap V_ 2$,
   $\Phi_ k^{(j)} (y) \cap O \neq \emptyset$,
   independently of which of \eqref{eq:def1}, \eqref{eq:def2}
   applies at $y$.

   2. If $ y_ 0 \in Y\setminus C_ k $, then we use that
   \eqref{eq:def2} applies on the open set $Y\setminus C_ k$ and
   that the set in \eqref{eq:hitspoint} is open with $M=O$.

   Hence, $\Phi_ k^{(j)}$ is lower semi-continuous and
   from \cite[Lemma 4.1]{Michael-CS1-1956}\footnote{\relax
             p.~368. $\mathcal K$ denotes (say, nonempty) convex sets (see p.~367).}
   it follows that there exists a~continuous function
   $\varphi_ k^{(j)} \fcolon Y \to Z $
   such that
\begin{equation}\label{eq:almostselection}
   \varphi_ k^{(j)} (y) \in  \Phi_ k^{(j)}(y) + B_ Z(0,2^{-k})
\end{equation}
   for every $y\in Y$.

   Obviously,
\begin{equation}\label{eq:approxOnCk}
	\normZ{\varphi_ k^{(j)} (y) - h_ k(y)} < 2^{-k}
\end{equation}
  for every $y\in C_ k$.

   Let $\tilde h_ k= \varphi_ k^{(k)}$.

   Now we want to show that 
   for every $y_ 0\in \bigcap_ i Y_ i$
   and $i_ 0 \in \N$, 
   there is $k_ 1$ and a~neighborhood $U$ of $y_ 0$ such that
\begin{equation}\label{eq:dzzzz}
   \normZ{
   \tilde h_ k(y) -
   f(y_ 0)
   }
   \le
   2^{2-i_ 0}
\end{equation}
   whenever $k\ge k_ 1$ and $y\in U$.

   Let $y_ 0\in \bigcap_ i Y_ i$.
   Since $Y_ {i_ 0}=\bigcup_ j Y_ {i_ 0}^{(j)}$,
   there is $j_ 0 \in \N$ such that
   $ y_ 0 \in Y_ {i_ 0}^{(j_ 0)}$.
   Thus,
   there is $G\in \GG_ {i_0}$ such that $y_ 0 \in \interior_ {1/j_0} G$.
   Obviously, $U:=B_ Y(y_ 0, 1/ 2j_0) \subset \interior_ {1/ 2j_0} G$.
   If $k\ge k_ 1 := \max( i_ 0, 2j_ 0)$, we have
   by
   \eqref{eq:almostselection}
   that,
   for every $y\in U$,
   the set
   $\tilde h_ k (y) + B_ Z(0, 2^{-k})$
   intersects
$$
   \Phi_ k^{(k)} (y)
\overset{\eqref{eq:PhiSubset}}
\subset
   \widetilde \Phi_ k^{(k)} (y)
\overset{\eqref{eq:m}}
\subset
   \widetilde \Phi_ {i_ 0}^{(2 j_ 0)} (y)
\overset{\eqref{eq:inball-F2}}
\subset
   \closure B_ Z( f(y_ 0), 2^{1-i_ 0} )
,
$$
and hence
$$
  \normZ{
   \tilde h_ k(y)
   -
   f(y_ 0)
  }
  \le
   2^{1-i_ 0}
   +
   2^{-k}
  \le
   2^{2-i_ 0}
    .
$$
   This finishes the proof of \eqref{eq:dzzzz}.

   Assume $y_ 0 \in Y$ and $h_ n(y_ 0) \to f(y_ 0)$.
   We want to show that $\tilde h_ n(y_ 0) \to f(y_ 0)$.

   1.
   If $y_ 0 \notin  \bigcap_ i Y_ i$, then there is $k_ 0$ such that 
   $y_ 0 \notin Y_ k$ for all $k\ge k_ 0$.
   We have
\begin{equation}\label{eq:constantB}
   \widetilde \Phi_ k(y_ 0) =  \widetilde \Phi_ {k_ 0} (y_ 0) =: B
   \qquad
   \text{ for all $k \ge k_ 0$}
   .
\end{equation}
   The set $B$  (defined by \eqref{eq:constantB}) is an intersection of a~finite number
   of (closed) balls with $f(y_ 0)$ in the interior --- recall
   the inclusion $f(G) \subset 
       B_ Z( z_ {k,G}, 2^{-k} ) 
   $ (see \eqref{eq:fG}).
   Since $h_ k(y_ 0) \to f(y_ 0)$, there is $k_ 1 \ge k_ 0$ such that,
   for all $k\ge k_ 1$, 
   $h_ k(y_ 0) \in B$, $y_0 \in C_ k$ (see \eqref{eq:constantB}) and $\normZ{ \tilde h_ k(y_ 0) - h_ k(y_ 0) } < 2^{-k}$
   by~\eqref{eq:approxOnCk}.
   Since $h_ k(y_ 0) \to f(y_ 0)$, we have also
   $\tilde h_ k(y_ 0) \to f(y_ 0)$.

   2. Assume that $y_ 0 \in \bigcap_ i Y_ i$ and $i_ 0 \in \N$ is given.
   Then we have \eqref{eq:dzzzz} for $y = y_ 0$ and for $k$ large enough.
   Since $i_ 0$ was arbitrary, we proved $\tilde h_ k(y_ 0) \to f(y_ 0)$.

   Assume now that $f$ is continuous at $y_ 0 \in Y$.
   Then $y_ 0\in \bigcap_ i Y_ i$
   by the definition of $\GGGG_ {k}$, $\GG_ k$ and $Y_k$.
   Therefore, $\tilde h_ k(y_ 0) \to f(y_ 0)$ by the previous paragraph.
   Moreover, 
   for every $i_ 0$, 
   there is $k_ 1$ and a~neighborhood $U$ of $y_ 0$ such that
   \eqref{eq:dzzzz} holds true for all $k\ge k_1$ and $y\in U$.
   Since $y_ 0$ was an arbitrary point of continuity of $f$,
   this proves
   that $ ( \{\tilde h_ k\}, f ) $ has property UCPC.
\end{proof}

\begin{lemma}\label{l:Lipschitz}
   Let $Y$ be a~metric space and $Z$ a~normed linear space.
   Given a~function $f\fcolon Y \to Z$ and
   a~sequence of continuous functions $h_ n\fcolon Y\to Z$,
   there exist bounded locally Lipschitz functions  $f_ n\fcolon Y \to Z$ such that
\begin{enumerate}[\textup(a\textup)]%
\item\label{item:l:Lipschitz:pointwise}
   $f_ n(y) \to z$ whenever $y\in Y$ and   $h_ n(y) \to z\in Z$;

\item\label{item:l:Lipschitz:UCPC}
   $(\{ f_ n \}, f)$
   has UCPC
   if
   $(\{ h_ n \}, f)$
   has UCPC;

\item\label{item:l:Lipschitz:unif-bounded}
\brijenxvbounded
   $\{ f_n \}$
   is uniformly bounded
   on $B_Y(a, r)$
   whenever
   $a\in Y$,
   $r\in (0,\infty)\cup\{\infty\}$,
   $h_n \to f$ pointwise on $B_Y(a,  2 r)$
   and $f$
   is bounded on $B_Y(a, 2 r)$.
\erijenxvbounded
\end{enumerate}
\end{lemma}

\begin{proof}
   The boundedness is very easy to achieve.
   Let  $\tilde h _ n = P_ n \circ h_ n$ where $P_ n$ is the radial projection
   of $Z$ onto $\closure B_ Z(0, n)$: 
\begin{align}\label{eq:radialprojection}
   P_ n(z)
   =
   \begin{cases}
     z ,              & z\in B_ Z(0, n) ;
\\
     n z / \normZ{ z } ,   & z \in Z \setminus B_ Z(0, n)
   .
   \end{cases}
\end{align}
   \bsrpenxv
   Every
   \esrpenxv
   $\tilde h_ n$ is obviously bounded while it retains other properties mentioned in
   Lemma~\ref{l:Lipschitz}(\ref{item:l:Lipschitz:pointwise}) and (\ref{item:l:Lipschitz:UCPC}).
   Henceforth we label them $h_n$ and assume they are bounded.

   Now we provide the local uniform boundedness
   property
\bcervenec
   requested by Lemma \ref{l:Lipschitz}(\ref{item:l:Lipschitz:unif-bounded})
\ecervenec
   which
   is slightly more complicated.
   \bsrpenxv
   (Note
   that this is needed only in supplementary parts of our application to differentiable
   extensions~\cite{KocKolarDer}.)
   \esrpenxvunskip

 For $n\in \N$, let
 $O_n = \interior \{ x \in Y \setcolon h_m(x) \to f(x) \text{ and } f(x) \in B_Z(0,n) \}$,
\bcervenec
 \begin{align*}
     \phi_n(x)
     &=
     \begin{cases}
        (n+1)  + 1/\dist(x,Y\setminus O_n)  &  \text{if }x\in O_n
        \\
        \infty   &  \text{if } x\in Y\setminus O_n
     ,
     \end{cases}
  \\
  \noalign{\noindent  and}
     r(x) &= \inf \{ \phi_n(x) \setcolon n \in \N\} \ge 1
     ,
     \qquad
     x\in Y
   .
 \end{align*}
\ecervenec

 Note that for every $x\in Y$ there is $n_0$ such that $\normZ{ h_n(x) } \le r(x)$ for every $n\ge n_0$.
 Indeed, if $h_n(x) \to f(x)$ and
 \brijenxv
 $m_0$
 \erijenxv
 is the least integer such that
 \brijenxv
 $f(x) \in B_Z(0,m_0)$
 then
 $x\in O_n$ for no $n < m_0$, hence
 (regardless if $x\in O_n$ or not, for various $n\ge m_0$)
 $$
  r(x)
  = \inf_{n\ge 1} \phi_n(x)
  = \inf_{n\ge m_0} \phi_n(x)
  \ge \inf_{n\ge m_0} n + 1
  =
  m_0 + 1 > \normZ{f(x)} \leftarrow \normZ{h_n(x)}
  .
 $$
 \erijenxv
 Otherwise,
 $r(x)=\infty$
 \bsrpenxv
 by the definition of $O_n$ and $\phi_n$.
 \esrpenxv

 If $r(x) < \infty$, then $r$ is bounded on a~neighborhood $U$ of $x$ (because each $\phi_n$ is continuous)
 and (since
 \bsrpen
 $\phi_n\ge n+1$)
 \esrpen
 there is $n_0$ such that
 $r = \min(\phi_1, \phi_2, \dots, \phi_{n_0})$ on $U$.
 Hence $r$ is continuous at $x$.
 If $r(x) = \infty$ then, for every $n\in \N$, $x\in Y \setminus O_n$ and hence $\phi_n(\cdot) \ge (n+1) + 1/\ddd(\cdot,x)$.
 Therefore $r(\cdot) \ge (n+1)
 + 1/\ddd(\cdot,x)$ and $r$ is again continuous at $x$.

 If $P_r$ ($r\ge 1$) is the radial projection from \eqref{eq:radialprojection}, then
 the map
 $(z,r) \to P_r(z)$ is
 $1$-Lipschitz on $Z\times [1,\infty)$
 and
 continuous on $Z\times [1,\infty]$.
 Indeed, its continuity at every point $(z,\infty)\in Z\times [1,\infty]$ is obvious.

 Letting
 \begin{equation}\label{eq:tilde-for-unifbdd}
 \tilde h_n(x) = P_{r(x)} (h_n(x))
 \end{equation}
 we obtain a~new sequence of functions that retains the boundedness
 \bkoneclistopaduxvH
 and
 \ekoneclistopaduxvH
 continuity
 properties
 of $\{h_n\}$.
 \bkoneclistopaduxvH
 We show that it also retains pointwise convergence
 as required by
 Lemma~\ref{l:Lipschitz}\itemref{item:l:Lipschitz:pointwise}.
 Assume that $x\in Y$, $z\in Z$
 and $h_n(x) \to z$.
 We need to check that
 $\tilde h_n(x) \to z$.
 If $r(x)=\infty$, it is obvious from~\eqref{eq:tilde-for-unifbdd}.
 If $r(x)<\infty$ then
 we need to look at the values of $\phi_n(x)$.
 If $n\in \N$ and $\phi_n(x)$ is finite
 then
 \ekoneclistopaduxvH
 necessarily $x\in O_n$,
 $f(x)=z$,
 \bkoneclistopaduxvH
 and hence
 \ekoneclistopaduxvH
 $\normZ{z} < n < \phi_n(x)$.
 \bkoneclistopaduxvH
 Thus we get
 $\normZ{z} < \phi_n(x)$
 for every $n$ and
 $\normZ{z} \le r(x)$.
 This concludes the proof of
 \ekoneclistopaduxvH
 $\tilde h_n(x) \to z$.

 Assume now that
 $(\{ h_n \}, f)$ has UCPC.
 We will prove that $(\{ \tilde  h_n \}, f)$ has UCPC.
 Assume that
 $f$ is continuous at $y_0\in Y$
 and
 $\varepsilon \in (0,1)$.
 Then
 $h_n(y_0) \to f(y_0)$ (cf.\ \eqref{eq:hkyfy0})
 and
 there is $k_0\in \N$ and a~neighborhood $U$ of $y_0$ such that
 \eqref{eq:hkyfy0} is true for $k\ge k_0$ and $y\in U$.
 We choose $n_0$ to be the least integer such that $f(y_0) \in B_Z(0,n_0)$.
 Then
 there is a~neighborhood $V$ of $y_0$ such
 that,
 \bsrpenxv
 for every $y\in V$,
 \esrpenxv
 $f(y) \in B_Z(0,n_0)\setminus B_Z(0, n_0-2)$,
 $y \notin O_{n_0-2}$
 and
 \begin{equation}\label{eq:MartinPodtrhl}
 r(y) \ge (n_0 - 1) + 1
 .
 \end{equation}
 For $y \in U\cap V$ and $k \ge \max(k_0, n_0)$
 we have $\tilde h_k (y) = h_k(y)$ and hence \eqref{eq:hkyfy0} is true also when
 $h_k$ is replaced by $\tilde h_k$.
 Hence $(\{ \tilde  h_n \}, f)$ has UCPC
 if $(\{  h_n \}, f)$  has UCPC.

\brijenxvbounded
   Finally we want to prove that
   $\{\tilde h_n\}$ from \eqref{eq:tilde-for-unifbdd}
   satisfies the
   \bkoneclistopaduxvH
   local
   \ekoneclistopaduxvH
   uniform boundedness
   property requested by Lemma \ref{l:Lipschitz}(\ref{item:l:Lipschitz:unif-bounded}).
  Assume
  that
  $a\in Y$,
  $r\in (0,\infty)\cup\{\infty\}$,
   $h_n \to f$ pointwise on
   $W:=B_Y(a, 2 r)$
   and
   there is $p_0\in \N$ such that
   $\normZ{ f(y) } < p_0$
   for all $y\in W$.
  Then obviously $W \subset O_{p_0}$.
  For every $y\in B_Y(a, r)$,
  we have
  $\dist(y, Y\setminus O_{p_0}) \ge r$
  and,
  by \eqref{eq:tilde-for-unifbdd} and \eqref{eq:radialprojection},
  $\normZ{ \tilde h_n(y) }
  \le
  r(y)
  \le
  \phi_{p_0}(y) \le M := p_0 + 1 + 1/r$.
\erijenxvbounded

 This closes the part of the proof
\brijenxv
  where we obtained $\{\tilde h_n\}$ with
\erijenxv
 the local uniform boundedness property
\bcervenec
   (Lemma~\ref{l:Lipschitz}(\ref{item:l:Lipschitz:unif-bounded}))
   while
   boundedness and the properties mentioned in
   Lemma~\ref{l:Lipschitz}(\ref{item:l:Lipschitz:pointwise}) and (\ref{item:l:Lipschitz:UCPC})
   are retained
\brijenxv
   by $\{\tilde h_n\}$.
\erijenxv
\ecervenec

 \medbreak

 Note that
 ``locally Lipschitz'' is the only property that we miss at this point.
   To replace the functions by locally Lipschitz ones
   is rather straightforward
   when the paracompactness of $Y$ is used.
   \brijenxv
   Though, formal argument takes at least several lines for each of the properties.
   \erijenxv

   Let
   \brijenxv
   $n\in \N$
   and let
   \erijenxv
   $\mathcal U_ n$
   be an open locally finite refinement
   of open cover
   $$
    \{B_ Y(x, \delta/2) \setcolon x\in Y,\ \delta\in(0,1/n)
   \text{ such that } \tilde h_ n(B_ Y(x,\delta)) \subset B_ Z(\tilde h_ n(x), 1/n) \}
   .
   $$
   Choose $x_ {n,U}\in U$ for every $U \in \mathcal U_ n$.
   Let $w_ {n,Y}(y)=1$ for all $y\in Y$ (or, suppose $Y\notin \mathcal U_ n$).
   For $y\in Y$ and $U\in \mathcal U_ n \setminus \{Y\}$ let 
\begin{align}
   \notag
   \mathcal U_ n(y) &= \{ V\in \mathcal U_ n \setcolon y \in V \},
   \displaybreak[0]
   \\\notag
   w_ {n,U}(y)
   =
   w_ U (y)
   &= \dist(y,Y\setminus U),
   \displaybreak[0]
   \\\label{eq:Wn}
   W_ n(y) &= \sum_ { U\in \mathcal U_ n(y) }
                w_ {n,U}(y),
   \\\label{eq:fn}
   f_ n(y) &= \sum_ { U\in \mathcal U_ n(y) }
             \,
             \frac{ w_ {n,U}(y) }{ W_ n(y) } \, \tilde h_ n(x_ {n,U})
	     .
\end{align}
   Then $W_ n$ is locally bounded away from zero,
   the sums in \eqref{eq:Wn}, \eqref{eq:fn} are locally finite,
   $w_ {n,U}$, $W_ n$ and $f_ n$ are locally Lipschitz and locally bounded.
   For every $y\in Y$,
   we have
$$ 
     f_ n(y) \in \cvx  \tilde h_ n\bigl( \{ x_ {n,U} \setcolon U \in \mathcal U_ n(y) \}
                         \bigr)
                \subset
                \cvx  \tilde h_ n\left( \textstyle \bigcup \mathcal U_ n(y) \right)
		\subset
		 B_ Z(\tilde h_ n(y), 2/n)
     ,
$$
   and thus
\begin{equation}\label{eq:blizko}
   \normZ{ f_ n (y) - \tilde h_ n  (y)} \le 2/n.
\end{equation}
\brijenxv
   Therefore
   $f_ n(y) \to z$
   whenever
   $h_ n(y) \to z \in Z$
   (which implies $\tilde h_ n(y) \to z$).
\erijenxv

\brijenxvbounded
   If $\tilde h_ n(y)$ are uniformly bounded on any given set, so are $f_ n$
   by \eqref{eq:blizko}.
   Since we already have validity of
   Lemma~\ref{l:Lipschitz}(\ref{item:l:Lipschitz:unif-bounded})
   for $\{\tilde h_n\}$, it is also valid for $\{f_n\}$
   defined by \eqref{eq:fn}.
\erijenxvbounded

   \brijenxv
   It remains to prove
   (\ref{item:l:Lipschitz:UCPC}).
   Assume $(\{h_n\},f)$ has UCPC.
   Then we already know that $(\{ \tilde h_n \}, f)$ has UCPC.
   \erijenxv
   Let $y_ 0\in Y$,
   $\varepsilon >0$,
   $k_ 0\in \N$ and $\delta > 0$
   be given
   such that
\begin{equation}\label{eq:epsdeltaN}
    \normZ{ \tilde h_ k(y) - f(y_ 0) } < \varepsilon
\end{equation}
   for every $k\ge k_ 0$ and $y \in B_ Y(y_ 0,\delta)$. 
   Let $n_ 0 \in \N$ satisfy
   $ 2/n_ 0 < \varepsilon $ and $n_ 0 \ge k_ 0$.
   Then,
   by \eqref{eq:blizko} and \eqref{eq:epsdeltaN},
\begin{equation}\label{eq:epsdeltaZ}
    \normZ{ f_ k(y) - f(y_ 0) } < 2 \varepsilon
\end{equation}
   for every $k\ge n_ 0$ and $y \in B_ Y(y_ 0,\delta)$. 
   Thus, $ ( \{f_ k\} , f ) $ has UCPC if $ ( \{h_ k\} , f ) $ has UCPC.
\bcervenec
   We see that we obtained locally Lipschitz functions
   while
   boundedness and the properties mentioned in
   Lemma~\ref{l:Lipschitz}(\ref{item:l:Lipschitz:pointwise}), (\ref{item:l:Lipschitz:UCPC})
   and (\ref{item:l:Lipschitz:unif-bounded})
   are retained.
\ecervenec
\end{proof}

\section{Extensions of Baire one functions}\label{sec:B1ext-}

Now we use Proposition~\ref{prop:Lip} and
\brijenxv
an elaborated refinement of
\erijenxv
the method of \cite[Theorem~6]{ALP} to prove Theorem~\ref{thm:B1ext}.
\begin{proof}[Proof of Theorem~\ref{thm:B1ext}]
   Applying  Proposition~\ref{prop:Lip} to $Y:=H$ and
   $f\fcolon H\to Z$, we get a~sequence $f_ n\fcolon H\to Z$ of
   bounded
   \JKblue locally \eJKblue 
   Lipschitz
   functions converging pointwise to $f$ on $H$
   \JKblue
   such that $ ( \{ f_n \} , f ) $ has property UCPC
   \eJKblue
\brijenxvbounded
   and such that $\{ f_n \}$
   is uniformly bounded
   on $B(a,6 r) \cap H$
   whenever
   $a\in H$,
   $r\in (0,\infty)\cup\{\infty\}$
   and $f$
   is bounded on $B(a, 12 r) \cap H$.
\erijenxvbounded

   Let $1\le M_ 1 \le M_ 2 \le \dots$ 
   be such that $\sup_{y\in H} \normZ{f_ n(y)}\le M_ n$.
   \JKblue
   For
   \eJKblue
   every $x\in X\setminus H$,
   we select a~point $u(x) \in H$ with $\ddd(x,u(x)) < 2\,\dist(x,H)$.
   \bkoneclistopaduxvH
   Then,
   for every $a\in H$ and $x\in X\setminus H$,
\begin{equation}\label{eq:generalInequality}
   \dist(x,H) \le \ddd(x,a)
   \qquad
   \text{and}
   \qquad
   \ddd(a,u(x)) \le 3 \ddd(a,x)
   .
\end{equation}
   Indeed,
   $
   \ddd(a, u(x))
   \le \ddd(a,x) + \ddd(x,u(x))
   \le \ddd(a,x) + 2 \dist(x,H)
   \le 3 \ddd(a,x)
   $.
   For $x\in X\setminus H$, let
   \ekoneclistopaduxvH
   \JKblue
\begin{equation}\label{eq:Kxn}
   K_ {x,n}=\max(1, \Lip f_ n |_ {B_ H(u(x), (n M_ n + 2)\dist(x,H))}), \qquad n\in \N
  .
\end{equation}
   Note that $K_ {x,n}$ might be infinite.
   Define $1/\infty=0$. %
   Let $n(x)$ be the largest $n\in \N$ such that 
\begin{equation}\label{eq:defnx}
   \dist(x,H) < ( n K_ {x,n} (n M_ n + 2) ) ^{-1} 
\end{equation}
   and let $n(x)=0$ if no such $n\in \N$ exists.
   Since $K_ {x,n}\ge 1$, $M_ n\ge 1$ and $\dist(x,H)>0$,
   there are only finitely many $n$ satisfying~\eqref{eq:defnx}.
   We claim that, for every
   \bkoneclistopaduxvH
   $a\in \boundary H$,
   \ekoneclistopaduxvH
\begin{equation}\label{eq:limnx}
   \lim_ {\substack{x\to a\\ x\in X\setminus H}} n(x)=\infty
   .
\end{equation}
   To show that, let
   \bkoneclistopaduxvH
   $a\in \boundary H$
   \ekoneclistopaduxvH
   and $n_ 0\in \N$ be fixed.
   Then there is $\eta > 0$ such that
$$
   K:=\max ( 1,  \Lip f_ {n_ 0} | _ {B_ H( a, ( n_ 0 M_ {n_ 0} + 2 + 3) \eta)} ) 
$$
   is finite. 
   \bkoneclistopaduxvH
   If $x\in X\setminus H$
   and
   \ekoneclistopaduxvH
\brijenxv
   $\ddd(x,a) < \eta$,
   \eqref{eq:generalInequality} shows that
$$
B_ H(u(x), (n_ 0 M_ {n_0} + 2)\dist(x,H))
\subset
B_ H( a, ( n_ 0 M_ {n_ 0} + 2 + 3) \eta)
$$
and
\erijenxv
   $K_ {x,n_ 0} \le K < \infty$.
   If, moreover, $\ddd(x,a) < \lambda := ( n_ 0 K (n_ 0 M_ {n_ 0} + 2) ) ^{-1}$,
   we see that \eqref{eq:defnx} is satisfied with $n=n_ 0$, 
   and hence $n(x)\ge n_ 0$. Therefore indeed $n(x)\ge n_ 0$  
   for all $x\in X\setminus H$ such that $\ddd(x,a) < \min( \eta, \lambda)$.

   Let $f_ 0(y)=0\in Z$ for $y\in H$ and
   define 
$$
   g(x) = f_ {n(x)}(u(x)),
   \qquad x\in X\setminus H 
  .
$$
   \eJKblue

\brijenxvbounded
   If
   $a\in H$,
   $r\in (0,\infty)\cup\{\infty\}$
   and $f$
   is bounded on $B(a, 12 r) \cap H$,
   we have
   $\{ f_n \}$
   uniformly bounded
   on $B(a,6 r) \cap H$.
   Then, by \eqref{eq:generalInequality}, $g$ is obviously bounded on $B_X(a, 2 r) \setminus H$.
   This proves property \eqref{eq:boundedness} for $g$, even with $B(a,r)$ replaced by $B(a,2r)$.
\erijenxvbounded

   We prove that if $a\in H$, $x\in X\setminus H$
   \JKblue
   and $n(x) > 0$
   \eJKblue
   then
   \begin{equation}\label{eq:ALP5}
   \normZ{g(x) - f(a)} \frac{\dist(x,H)}{\ddd(x,a)} \le \frac 1{\JKblue n(x) \eJKblue} +
       \frac{\normZ{f(a)}}{\JKblue n(x) \eJKblue} +
       \normZ{f_ {\JKblue n(x) \eJKblue}(a) - f(a)}
       .
   \end{equation}
   Since $f_ n(a) \to f(a)$
\JKblue
   and \eqref{eq:limnx} is true
   \bkoneclistopaduxvH
   for every $a\in \boundary H$,
   \ekoneclistopaduxvH
\eJKblue
   this will prove \eqref{eq:ALP3}.
\JKblue 
   Denote $n_ x:=n(x)$.
\eJKblue
   We distinguish between
   two cases.

   If $\frac{\dist(x,H)}{\ddd(x,a)} \le \frac 1{ n_ x M_ { n_ x }}$ then we have 
      $\normZ{g(x) - f(a)}\frac{\dist(x,H)}{\ddd(x,a)} \le \normZ{f_ { n_ x }(u(x)) - f(a)} \cdot
   \frac 1{ n_ x M_ { n_ x }} \le \frac{M_ { n_ x }}{ n_ x M_ { n_ x }} + \frac{\normZ{f(a)}}{ n_ x M_ { n_ x }}$
   and thus \eqref{eq:ALP5} holds true.

   If $\frac{\dist(x,H)}{\ddd(x,a)} > \frac 1{ n_ x M_ { n_ x }}$ then
\begin{multline}\label{eq:rho-ux-a}
   \ddd(u(x),a) \le 
   \ddd(u(x),x) +\ddd(x,a) <
   2 \,\dist(x,H) + n_ x M_ { n_ x }\!\dist(x,H)
   = ( n_ x M_ { n_ x }+2)\,\dist(x,H) <
   1/ ( n_ x K_ {\JKblue x, \eJKblue n_ x } )
\end{multline}
   by~\eqref{eq:defnx}, 
   and hence
   \begin{gather*}
      \normZ{g(x)-f(a)} = \normZ{ f_ { n_ x }(u(x)) - f(a) } \le \normZ{ f_ { n_ x }(u(x)) - f_ { n_ x }(a) } + \\
      + \normZ{f_ { n_ x }(a) - f(a)}
    \overset{\JKblue \eqref{eq:Kxn}\eJKblue}
      \le
      K_ {x, n_ x }\,\ddd(u(x),a) + \normZ{ f_ { n_ x }(a) - f(a) }
      \overset{\eqref{eq:rho-ux-a}}{<}
      \frac{1}{n_ x} + \normZ{ f_ { n_ x }(a) - f(a)}.
   \end{gather*}
   Since $\dist(x,H) \le \ddd(x,a)$, this implies \eqref{eq:ALP5}.
   This completes the proof
   of
   \eqref{eq:ALP3}.

   \JKblue
 
   Now we want to prove \eqref{eq:continuity}.
   Suppose that $f$ is continuous at $ a \in \boundary H$.
   Let $\varepsilon > 0$ be given.
   Applying the property UCPC of $\{f_ k\}$ at $y_ 0=a$ (see \eqref{eq:hkyfy0})
   we obtain
   $\delta>0$ and $k_ 0\in \N$ such that
\begin{equation}\label{eq:epsdelta-00}
    \normZ{ f_ k(y) - f( a ) } < \varepsilon
\end{equation}
   for every $k\ge k_ 0$ and $y \in B_ H( a ,\delta)$. 
   By~\eqref{eq:limnx}, there is $\delta_ 1>0$ such that
   $n(x) \ge n_ 0$ whenever $x\in X\setminus H$ and $\ddd(x, a )<\delta_ 1$.
   Now, if $x \in X\setminus H$ and $\ddd(x, a )< \min(\delta_ 1,\delta)/3$ 
   then,
\brijenxv
   by \eqref{eq:generalInequality},
\erijenxv
    $\ddd( u(x), a )
\brijenxv
     \le 3 \ddd( x, a )
\erijenxv
     <   \min(\delta_ 1,\delta)
    $
   and therefore, by \eqref{eq:epsdelta-00},
$$
   \normZ{ g(x) - f( a ) }
   = \normZ{ f_ {n(x)} (u(x)) - f( a ) }
   < \varepsilon
  .
$$
   Therefore, we have \eqref{eq:continuity} whenever $ a \in \boundary H$
   and $f$ is continuous at $ a $.
\eJKblue

\bledenredcontinuous
   So far $g$ has all required properties but being continuous.
   Let
   \[
           \mathcal U=
           \{ B_{X\setminus H} (x, \dist(x, H) /
           \bkoneclistopaduxvH
           3
           \ekoneclistopaduxvH
           ) \setcolon x\in X\setminus H \}
           .
   \]
   Using the paracompactness of $(X\setminus H, \ddd)$, let
   $\{ \phi_\alpha\}_ {\alpha\in \AAA}$ be a~continuous
   locally finite
   partition of unity subordinated to $\mathcal U$,
\brijenxv
   (see e.g.,~\cite[Theorem VIII.4.2]{DuguTopo}).
\erijenxv
   For $\alpha \in \AAA$, find $x_\alpha \in X\setminus H$ such that
   $\spt \phi_\alpha \subset B_{X\setminus H} (x_\alpha, \dist(x_\alpha, H)
   /
   \bkoneclistopaduxvH
   3
   \ekoneclistopaduxvH
   )$.
   Let
   \[
      \tilde g (x) = \sum _ {\alpha \in \AAA} \phi_\alpha (x)  g(x_\alpha)
      \qquad
      x\in X\setminus H
      .
   \]
   Then $\tilde g$ is continuous on $X \setminus H$.
   Whenever $g$ satisfies \eqref{eq:continuity} resp.\ \eqref{eq:boundedness},
   the same is true for $\tilde g$
\brijenxvbounded
   (with $B(a, 2 r) \setminus H$ for $g$
   replaced by $B(a, r) \setminus H$ for $\tilde g$).
\erijenxvbounded
   If $\alpha \in \AAA$, $x\in B_{X\setminus H} (x_\alpha, \dist(x_\alpha, H)
   /
   \bkoneclistopaduxvH
   3
   \ekoneclistopaduxvH
   )$
   and
   \bkoneclistopaduxvH
   $a\in \boundary H$
   \ekoneclistopaduxvH
   then
   \[
      \frac {1}{4}
      \frac { \dist(x, H) }{ \ddd(x, a) }
      \le
      \frac { \dist(x_\alpha, H) }{ \ddd(x_\alpha, a) }
      \le
      4
      \frac { \dist(x, H) }{ \ddd(x, a) }
      .
   \]
   Hence from \eqref{eq:ALP3} for $g$ we obtain that
    \eqref{eq:ALP3} is also true for $\tilde g$.
\eledenredcontinuous
\end{proof}

\end{document}